\documentclass[final,3p, times, 12pt]{elsarticle}

\usepackage[centertags]{amsmath}
\usepackage{amsfonts}
\usepackage{amssymb}
\usepackage{amsthm}

\usepackage{graphicx}
\usepackage{multicol}

\theoremstyle{plain}
\newtheorem{thm}{Theorem}[section]
\newtheorem{cor}[thm]{Corollary}
\newtheorem{lem}[thm]{Lemma}

\newtheorem{conj}[thm]{Conjecture}

\theoremstyle{definition}
\newtheorem{defn}[thm]{Definition}
\newtheorem{exam}[thm]{Example}
\newtheorem{rem}[thm]{Remark}

\theoremstyle{remark}
\numberwithin{equation}{section}

\newproof{pf}{Proof}

\begin{document}

\begin{frontmatter}

\title{Spectral Radius of Bipartite Graphs}

\author[CL]{Chia-an Liu\corref{cor}}\ead{twister.imm96g@g2.nctu.edu.tw}
\author[CW]{Chih-wen Weng}\ead{weng@math.nctu.edu.tw}

\date{February 23, 2014}
\cortext[cor]{Corresponding author}

\address[CL]{Department of Applied Mathematics, National Chiao-Tung University, Hsinchu, Taiwan}
\address[CW]{Department of Applied Mathematics, National Chiao-Tung University, Hsinchu, Taiwan}

\begin{abstract} Let $k, p, q$  be positive integers with $k<p<q+1$.
  We prove that the maximum spectral radius of a simple bipartite graph obtained
 from the complete bipartite graph $K_{p, q}$ of bipartition orders $p$ and $q$ by deleting $k$ edges is attained when the deleting edges are all incident on a common vertex which is located in the partite set of order $q$.
Our method is based on new sharp upper bounds on the spectral radius of bipartite graphs in terms of their degree sequences.
\end{abstract}

\end{frontmatter}


\section{Introduction}   \label{s1}

Let $G$ be a simple graph of order $n.$
The {\it adjacency matrix} $A=(a_{ij})$ of $G$ is a binary square matrix of order $n$
with rows and columns indexed by the vertex set $VG$ of $G$ such that
for any $i, j\in VG$, $a_{ij}=1$ iff $i,j$ are adjacent in $G.$
The \emph{spectral radius} $\rho(G)$ of $G$ is the largest eigenvalue of the adjacency matrix $A$ of $G.$

\medskip

Braualdi and Hoffman proposed the problem of finding
the maximum spectral radius of a graph with precisely $e$ edges in $1976$
\cite[p.438]{bfvs:76}, and ten years later
they gave a conjecture in \cite{ra:85} that the maximum spectral radius of a graph with
$e$ edges is attained by taking a complete graph and adding a new vertex which is
adjacent to a corresponding number of vertices in the complete graph.
This conjecture was proved by Peter Rowlinson in \cite{p:88}.
See \cite{r:87, s:88} also for the proof of partial cases of this conjecture.
\medskip

The next problem is then to determine graphs with maximum spectral radius in the class of connected graphs with $n$ vertices and $e$ edges. The cases $e\leq n+5$ are settled by
Brualdi and Solheid \cite{bs:86}, and the cases $e-n={r\choose 2}-1$ by F. K. Bell \cite{b:91}.

\medskip

The bipartite graphs analogue of the Brualdi-Hoffman conjecture was settled by
A. Bhattacharya, S. Friedland, U.N. Peled \cite{bf:08} with the following statement: For a connected bipartie graph $G$, $\rho(G)\leq \sqrt{e}$ with equality iff $G$ is a complete bipartite graph. Moreover, they
proposed the problem
to determine graphs with maximum spectral radius in the class of bipartite graphs with bipartition orders $p$ and $q$, and $e$ edges. They then gave Conjecture \ref{conj1} below.
\medskip

From now on the graphs considered are simple bipartite.
Let $\mathcal{K}(p,q,e)$ denote the family of $e$-edge subgraphs of the complete bipartite graph $K_{p,q}$ with bipartition orders $p$ and $q$.

\begin{conj}  \label{conj1}
Let $1 < e < pq$ be integers. An extremal graph that solves
$$ \max_{G \in \mathcal{K}(p,q,e)} \rho(G) $$
is obtained from a complete bipartite graph by adding one vertex and
a corresponding number of edges.
\end{conj}

\medskip

Conjecture~\ref{conj1} does not indicate that the adding vertex goes into which partite set of a complete bipartite graph. Let $K_{p, q}^{[e]}$ (resp. $K_{p, q}^{\{e\}}$) denote the graph which is obtained
from $K_{p, q}$ by deleting $pq-e$ edges incident on a common vertex in the partite set
of order no larger than (resp. no less than) that of the other partite set.
Figure 1 gives two such graphs.
\bigskip

\begin{center}
\begin{multicols}{2}
\begin{picture}(50,80)
\put(10,40){\circle*{3}} \put(10,60){\circle*{3}}
\put(40,30){\circle*{3}} \put(40,50){\circle*{3}}
\put(40,70){\circle*{3}}
\qbezier(10,60)(25,65)(40,70) \qbezier(10,60)(25,55)(40,50)
\qbezier(10,60)(25,45)(40,30) \qbezier(10,40)(25,55)(40,70)
\qbezier(10,40)(25,45)(40,50)
\put(-5, -10){$K_{2, 3}^{\{5\}} = K_{2,3}^{[5]}$}
\end{picture}

\begin{picture}(50,80)
 \put(10,40){\circle*{3}} \put(10,60){\circle*{3}}
\put(40,10){\circle*{3}} \put(40,30){\circle*{3}} \put(40,50){\circle*{3}}
\put(40,70){\circle*{3}}
\qbezier(10,60)(25,65)(40,70) \qbezier(10,60)(25,55)(40,50)
\qbezier(10,60)(25,45)(40,30) \qbezier(10,60)(25,35)(40,10)
\qbezier(10,40)(25,55)(40,70)
\put(10, -10){$K_{2, 4}^{[5]}$}
\end{picture}
\end{multicols}
\bigskip

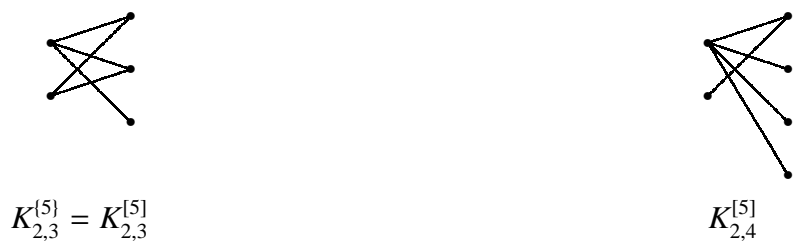
{\bf Figure 1:} The graphs $K_{2, 3}^{\{5\}}$, $K_{2,3}^{[5]}$ and $K_{2, 4}^{[5]}$.
\end{center}

\bigskip

Since the number $pq-e$ of deleting edges is at most $\max(p, q)$ in $K_{p, q}^{[e]}$,
the constraint $e \geq pq-\max(p,q)$ is implicitly assumed when the notation $K_{p,q}^{[e]}$ is used,
and similarly for the constraint $e \geq pq-\min(p,q)$ in $K_{p,q}^{\{e\}}.$
Then the extremal graph in Conjecture~\ref{conj1} is either $K_{s, t}^{[e]}$ or $K_{s, t}^{\{e\}}$
for some positive integers $s \leq p$ and $t \leq q$ which meet the above constraints.
In 2010 \cite{cfksw:10}, Yi-Fan Chen, Hung-Lin Fu, In-Jae Kim,
Eryn Stehr and Brendon Watts determined $\rho(K_{p,q}^{\{e\}})$
and gave an affirmative answer to Conjecture~\ref{conj1} when $e=pq-2.$
Furthermore, they refined Conjecture~\ref{conj1} for the case when the number
of edges is at least $pq-\min(p,q)+1$ to the following conjecture.
\bigskip

\begin{conj}  \label{conj2}
Suppose $0 < pq-e < \min(p,q)$. Then for $G \in \mathcal{K}(p,q,e),$
\begin{equation}
\rho(G) \leq \rho\left(K_{p, q}^{\{e\}}\right).
\label{fusbd} \nonumber
\end{equation}
\end{conj}

\medskip

\medskip

The paper is organized as follows.
Preliminary contents are in Section~\ref{sp}.
Theorem~\ref{thm3} in Section~\ref{s2}  presents
a series of sharp upper bounds of $\rho(G)$ in terms of
the degree sequence of $G.$
Some special cases of Theorem~\ref{thm3}
are further investigated in Section~\ref{sapp} with which Corollary~\ref{corphipq} is the most useful in this paper.
We prove Conjecture~\ref{conj2} as an application of Corollary~\ref{corphipq} in Section~\ref{s_conj}.
Finally we propose another conjecture which is a general refinement of Conjecture~\ref{conj1}
in Section~\ref{concluding}.

\section{Preliminary}   \label{sp}

Basic results are provided in this section for later used.
\medskip

\begin{lem}(\cite[Proposition~2.1]{bf:08})   \label{lem_sqrte}
Let $G$ be a simple bipartite graph  with $e$ edges. Then
$$\rho(G) \leq \sqrt{e}$$ with equality
iff $G$ is a disjoint union of a complete bipartite graph and isolated vertices.
\end{lem}

Let $G$ be a simple bipartite graph  with bipartition orders $p$ and $q$, and degree sequences $d_{1} \geq d_{2} \geq \cdots \geq d_{p}$ and $d'_{1} \geq d'_{2} \geq \cdots \geq d'_{q}$ respectively.
We say that $G$ is \emph{biregular} if $d_{1}=d_{p}$ and $d'_{1}=d'_{q}.$

\begin{lem}(\cite[Lemma~2.1]{bz:01})  \label{lem_d1d1}
Let $G$ be a simple connected  bipartite graph. Then
$$\rho(G) \leq \sqrt{d_{1}d'_{1}}$$
with equality iff $G$ is biregular.
\end{lem}

Let $M$ be a real matrix described in the following block form
$$M=\left(
      \begin{array}{ccc}
        M_{1,1} & \cdots & M_{1,m} \\
        \vdots &  & \vdots \\
        M_{m,1} & \cdots & M_{m,m} \\
      \end{array}
    \right),$$
where the diagonal blocks $M_{i, i}$ are square.
Let $b_{i,j}$ denote the average row-sums of $M_{i,j},$ i.e. $b_{i, j}$ is the sum of entries in $M_{i, j}$ divided by the number of rows.
Then $B=(b_{i,j})$ is called a {\it quotient matrix} of $M$. If in addition
for each pair $i, j$, $M_{i,j}$ has constant row-sum, then  $B$ is called an \emph{equitable quotient matrix} of $M$. The following lemma is direct from the definition of matrix multiplication~\cite[Chapter 2]{aw:11}.

\begin{lem}     \label{lemeqpa}
Let $B$ be an equitable quotient matrix of $M$ with an eigenvalue $\theta.$
Then $M$ also has the eigenvalue $\theta.$
\end{lem}


\medskip

The following lemma is a part of the Perron-Frobenius Theorem~\cite[Chapter 2]{m:88}.

\begin{lem}   \label{lem4}
If $M$ is a nonnegative $n \times n$ matrix with largest eigenvalue $\rho(M)$ and row-sums
$r_{1}, r_{2},\ldots,r_{n},$ then
\begin{equation}
\rho(M) \leq \max_{1 \leq i \leq n}r_{i}.   \nonumber
\end{equation}
Moreover, if $M$ is irreducible then the above equality holds
if and only if the row-sums of $M$ are all equal.
\end{lem}

\section{A series of sharp upper bounds of $\rho(G)$}   \label{s2}

We give a series of sharp upper bounds of $\rho(G)$ in terms of the degree sequence of
a bipartite graph $G$ in this section.
The following set-up is for the description of extremal graphs of our upper bounds.
\medskip

\begin{defn}
Let $H,$ $H'$ be two bipartite graphs with given ordered bipartitions
$VH=X \cup Y$ and $VH'=X' \cup Y'.$ The \emph{bipartite sum} $H+H'$ of $H$ and $H'$
(with respect to the given ordered bipartitions)
is the graph obtained from $H$ and $H'$ by adding an edge between $x$ and $y$ for each pair
$(x,y) \in X \times Y' \cup X' \times Y.$
\end{defn}

\medskip

\begin{exam}   \label{exam_almostKpq}
Let $N_{s,t}$ denote the bipartite graph with bipartition orders $s,t$ and without any edges.
Then for $p\leq q$ and $e$ meeting desired constraint,
$K_{p,q}^{[e]}=K_{p-1, q-pq+e}+N_{1, pq-e}$
and $K_{p,q}^{\{e\}}=K_{p-pq+e, q-1}+N_{pq-e, 1}.$
\end{exam}

\begin{thm}  \label{thm3}
Let $G$ be a simple bipartite graph with bipartition  orders $p$ and $q$,
and corresponding degree sequences
$d_{1} \geq d_{2} \geq \cdots \geq d_{p}$ and $d'_{1} \geq d'_{2} \geq \cdots \geq d'_{q}.$
For $1 \leq s \leq p$ and $1\leq t\leq q,$  let
$X_{s,t}=d_{s}d'_{t}+\sum_{i=1}^{s-1}(d_{i}-d_{s})+\sum_{j=1}^{t-1}(d'_{j}-d'_{t})$
and $Y_{s, t}=\sum_{i=1}^{s-1}(d_{i}-d_{s}) \cdot \sum_{j=1}^{t-1}(d'_{j}-d'_{t}).$
Then the spectral radius
$$\rho(G) \leq \phi_{s,t}:=\sqrt{\frac{X_{s, t}+\sqrt{X_{s, t}^{2}-4Y_{s, t}}}{2}}.$$
Furthermore, if $G$ is connected then the above equality holds if and only if
there exists nonnegative integers $s'<s$ and $t'<t,$
and a biregular graph $H$ of bipartition orders $p-s'$ and $q-t'$ respectively
such that $G=K_{s',t'}+H.$
\end{thm}

\medskip

Before proving Theorem~\ref{thm3}, we mention some simple properties of  $\phi_{s,t}.$

\medskip

\begin{lem}  \label{rem2}
\begin{enumerate}
\item[(i)] $\phi_{1, 1} = \sqrt{d_1d'_1}.$
\item[(ii)] If $d_{s'}=d_{s}$ then $\phi_{s', t}=\phi_{s, t}.$
            If $d'_{t'}=d'_{t}$ then $\phi_{s, t'}=\phi_{s, t}.$
\item[(iii)]
\begin{equation}
\phi^{2}_{s,t} \geq \max \left( \sum_{i=1}^{s-1}(d_{i}-d_{s}), \sum_{j=1}^{t-1}(d'_{j}-d'_{t}) \right)
\nonumber
\end{equation}
with equality iff $\phi^{2}_{s,t}=e.$
\item[(iv)]
$\phi^{4}_{s,t} - X_{s,t}\phi^{2}_{s,t} + Y_{s,t} = 0.$
\end{enumerate}
\end{lem}

\begin{proof}
(i), (ii), (iv) are immediate from the definition of $\phi_{s, t}$.
Clearly $d_{s}d'_{t}=0$ if and only if
$$\max \left( \sum_{i=1}^{s-1}(d_{i}-d_{s}), \sum_{j=1}^{t-1}(d'_{j}-d'_{t}) \right)=e.$$
Hence (iii) follows by using that
$X_{s,t} \geq \sum_{i=1}^{s-1}(d_{i}-d_{s})+\sum_{j=1}^{t-1}(d'_{j}-d'_{t})$ with equality iff $d_{s}d'_{t}=0$
to simplify $\phi_{s,t}$.

\end{proof}

\medskip

We set up notations for the use in the proof of  Theorem~\ref{thm3}.
For $1 \leq k \leq s-1$, let
\begin{equation}
x_{k} =
\left\{
\begin{array}{ll}
  1+\displaystyle\frac{d'_{t}(d_{k}-d_{s})}{\phi_{s, t}^{2}-\sum_{i=1}^{s-1}(d_{i}-d_{s})} &
  , \text{ if } \phi_{s, t}^{2} > \sum_{i=1}^{s-1}(d_{i}-d_{s});\\
  1, & \text{ if } \phi_{s, t}^{2} = \sum_{i=1}^{s-1}(d_{i}-d_{s}),
\end{array}
\right.
\label{eq2.5}
\end{equation}
and for $1\leq \ell \leq t-1$ let
\begin{equation}
x'_{\ell} =
\left\{
\begin{array}{ll}
  1+\displaystyle\frac{d_{s}(d'_{\ell}-d'_{t})}{\phi_{s, t}^{2}-\sum_{j=1}^{t-1}(d'_{j}-d'_{t})}, &
   \text{ if } \phi_{s, t}^{2} > \sum_{j=1}^{t-1}(d'_{j}-d'_{t});\\
  1, &  \text{ if } \phi_{s, t}^{2} = \sum_{j=1}^{t-1}(d'_{j}-d'_{t}).
\end{array}
\right.
\label{eq2.6}
\end{equation}
Note that $x_{k}, x'_{\ell} \geq 1$ because of Lemma~\ref{rem2}(iii).
The relations between the above parameters are given in the following.

\medskip

\begin{lem}   \label{lem_x}
\begin{enumerate}
\item[(i)] Suppose $\phi^{2}_{s,t}>\sum_{a=1}^{s-1}(d_{a}-d_{s}).$
Then
\begin{equation}
\frac{1}{x_{i}} \left(  d_{i}d'_{t} + \sum_{h=1}^{t-1}(d'_{h}-d'_{t})
+ \sum_{k=1}^{s-1}(x_{k}-1)d_{i}  \right) = \phi^{2}_{s,t}
\nonumber
\end{equation}
for $1 \leq i \leq s-1,$ and
\begin{equation}
d_{s}d'_{t}+\sum_{h=1}^{t-1}(d'_{h}-d'_{t})
+ \sum_{k=1}^{s-1}(x_{k}-1)d_{s} = \phi^{2}_{s,t}.
\nonumber
\end{equation}
\item[(ii)] Suppose $\phi^{2}_{s,t}>\sum_{b=1}^{t-1}(d'_{b}-d'_{t}).$
Then
\begin{equation}
\frac{1}{x'_{j}} \left(  d_{s}d'_{j} + \sum_{h=1}^{s-1}(d_{h}-d_{s})
+ \sum_{\ell=1}^{t-1}(x'_{\ell}-1)d'_{j}  \right) = \phi^{2}_{s,t}
\nonumber
\end{equation}
for $1 \leq j \leq t-1,$ and
\begin{equation}
d_{s}d'_{t}+\sum_{h=1}^{s-1}(d_{h}-d_{s})
+ \sum_{\ell=1}^{t-1}(x'_{\ell}-1)d'_{t} = \phi^{2}_{s,t}.
\nonumber
\end{equation}
\end{enumerate}
\end{lem}

\begin{proof}
Referring to \eqref{eq2.5} and Lemma~\ref{rem2}(iv),
\begin{eqnarray}
&&\frac{1}{x_{i}} \left( d_{i}d'_{t}+\sum_{h=1}^{t-1}(d'_{h}-d'_{t})
+ \sum_{k=1}^{s-1}(x_{k}-1)d_{i}  \right)
\nonumber   \\
&=& \frac{1}{\phi^{2}_{s,t}-\sum_{k=1}^{s-1}(d_{k}-d_{s})+d'_{t}(d_{i}-d_{s})}
\left(
\phi^{2}_{s,t}\left(d_{i}d'_{t}+\sum_{h=1}^{t-1}(d'_{h}-d'_{t})\right)
-\sum_{h=1}^{t-1}(d'_{h}-d'_{t})\sum_{k=1}^{s-1}(d_{k}-d_{s})
\right)
\nonumber   \\
&=& \phi^{2}_{s,t}
\nonumber
\end{eqnarray}
for $1 \leq i \leq s-1,$ and
\begin{eqnarray}
&& d_{s}d'_{t} + \sum_{h=1}^{t-1}(d'_{h}-d'_{t}) + \sum_{k=1}^{s-1}(x_{k}-1)d_{s}
\nonumber   \\
&=& \frac{1}{\phi^{2}_{s,t}-\sum_{k=1}^{s-1}(d_{k}-d_{s})}
\left(
\phi^{2}_{s,t}\left(d_{s}d'_{t}+\sum_{h=1}^{t-1}(d'_{h}-d'_{t})\right)
-\sum_{h=1}^{t-1}(d'_{h}-d'_{t})\sum_{k=1}^{s-1}(d_{k}-d_{s})
\right)
\nonumber \\
&=& \phi^{2}_{s,t}.
\nonumber
\end{eqnarray}
Hence (i) follows.
Similarly, referring to \eqref{eq2.6} and Lemma~\ref{rem2}(iv) we have (ii).
\end{proof}
\medskip

Let $U=\{u_i~|~1\leq i\leq p\}$ and $V=\{v_j~|~1\leq j\leq q\}$ be bipartition of $G$ such that the degree sequences
$d_{1} \geq d_{2} \cdots \geq d_{p}$ and $d'_{1} \geq d'_{2} \cdots \geq d'_{q},$
respectively are according to the list.
For $1 \leq i,j \leq p,$ let $n_{ij}$ denote the numbers of common neighbors of $u_{i}$ and $u_{j},$ i.e.,
$n_{ij} = |G(u_{i}) \cap G(u_{j})|$ where $G(u)$ is the set of neighbors of the vertex $u$ in $G.$
Similarly, for $1 \leq i,j \leq q$ let $n'_{ij} = |G(v_{i}) \cap G(v_{j})|.$
Since $G$ is bipartite, the adjacency matrix $A$ and its square $A^2$ look like the following in block form:
\begin{equation}
A=\left(
         \begin{array}{cc}
           0 & B \\
           B^{T} & 0 \\
         \end{array}
       \right),\quad
A^{2} = \left(
           \begin{array}{cc}
             BB^{T} & 0 \\
             0 & B^{T}B \\
           \end{array}
         \right)=
         \left(
           \begin{array}{cc}
             (n_{ij})_{1 \leq i,j \leq p} & 0 \\
             0 & (n'_{ij})_{1 \leq i,j \leq q} \\
           \end{array}
         \right).
\label{e}
\end{equation}
 We have the following properties of $n_{ij}$ and $n'_{ij}$.

\medskip

\begin{lem} \label{rem3}
\begin{enumerate}
\item[(i)] For $1 \leq i \leq p$ and $1 \leq j \leq q,$ $n_{ii}=d_{i}$ and $n'_{jj}=d'_{j}.$
\item[(ii)] For $1 \leq i, j \leq p,$ $n_{ij} \leq d_{i}$ with equality if and only if
$G(u_{j}) \supseteq G(u_{i}).$
\item[(iii)] For $1 \leq i, j \leq q,$ $n'_{ij} \leq d'_{i}$ with equality if and only if
$G(v_{j}) \supseteq G(v_{i}).$
\item[(iv)] For $1 \leq i \leq p,$
$$\sum_{k=1}^{p}n_{ik} = \sum_{j:~u_{i}v_{j}\in EG}d'_{j}
\leq (d_{i}-t+1)d'_{t} + \sum_{h=1}^{t-1}d'_{h} .$$
\item[(v)] For $1 \leq j \leq q,$
$$\sum_{k=1}^{q}n'_{jk} = \sum_{i:~u_{i}v_{j}\in EG}d_{i}
\leq (d'_{j}-s+1)d_{s} + \sum_{h=1}^{s-1}d_{h}.$$
\end{enumerate}
\end{lem}
\begin{proof}
(i)-(iii) are immediate from the definition of $n_{ij}$.
Counting the pairs $(u_k, v_j)$ such that $v_j\in G(u_i)\cap G(u_k)$  in two orders $(j, k)$ and $(k, j)$, we have the first equality in (iv).
The second inequality of (iv) is clear since $d'_j$ is non-increasing. (v) is similar to (iv).
\end{proof}

\noindent{\bf The proof of Theorem~\ref{thm3}}
\begin{proof}
Clearly $\rho(A)^2\leq \rho(A^2).$
In the following we will show that $\rho(A^{2}) \leq \phi_{s,t}^{2}.$
Let $$U = diag(\underbrace{x_{1}, x_2, \cdots, x_{s-1}, 1, \cdots, 1}_{p},
\underbrace{x'_{1}, x'_2, \cdots, x'_{t-1}, 1, \cdots, 1}_{q})$$
be a diagonal  matrix of order $p+q.$ Let
$C = U^{-1} A^2 U.$ Then $A^2$ and $C$ are similar and with the same spectrum.
Let $r_{1},\cdots,r_{p},r'_{1},\cdots,r'_{q}$ be the row-sums of $C.$
Referring to \eqref{e}, we have
\begin{eqnarray}
r_{i} &=& \sum_{k=1}^{s-1}\frac{x_{k}}{x_{i}}n_{ik} + \sum_{k=s}^{p}\frac{1}{x_{i}}n_{ik}
= \frac{1}{x_{i}}\sum_{k=1}^{p}n_{ik} + \frac{1}{x_{i}}\sum_{k=1}^{s-1}(x_{k}-1)n_{ik}
~~~\text{for}~~~ 1 \leq i \leq s-1;
\label{eq2.1}   \\
r_{i} &=& \sum_{k=1}^{s-1}x_{k}n_{ik} + \sum_{k=s}^{p}n_{ik}
=\sum_{k=1}^{p}n_{ik} + \sum_{k=1}^{s-1}(x_{k}-1)n_{ik}
~~~\text{for}~~~ s \leq i \leq p;
\label{eq2.2}   \\
r'_{j} &=& \sum_{\ell=1}^{t-1}\frac{x'_{\ell}}{x'_{j}}n'_{j\ell} + \sum_{\ell=t}^{q}\frac{1}{x'_{j}}n'_{j\ell}
= \frac{1}{x'_{j}}\sum_{\ell=1}^{q}n'_{j\ell} + \frac{1}{x'_{j}}\sum_{\ell=1}^{t-1}(x'_{\ell}-1)n'_{j\ell}
~~~\text{for}~~~ 1 \leq j \leq t-1;
\label{eq2.3}   \\
r'_{j} &=& \sum_{\ell=1}^{t-1}x'_{\ell}n'_{j\ell} + \sum_{\ell=t}^{q}n'_{j\ell}
=\sum_{\ell=1}^{q}n'_{j\ell} + \sum_{\ell=1}^{t-1}(x'_{\ell}-1)n'_{j\ell}
~~~\text{for}~~~ t \leq j \leq q.
\label{eq2.4.3}
\end{eqnarray}
If $\phi^{2}_{s,t}=\sum_{a=1}^{s-1}(d_{a}-d_{s})$ then $x_{k}=1$ for $1 \leq k \leq s-1$
by \eqref{eq2.5} and $\phi^{2}_{s,t}=e$ by Lemma~\ref{rem2}(iii).
Hence \eqref{eq2.1} and \eqref{eq2.2} become
\begin{equation}
r_{i} = \sum_{k=1}^{p}n_{ik} = \sum_{j:~u_{i}v_{j}\in EG}d'_{j} \leq e = \phi^{2}_{s,t}
\end{equation}
for $1 \leq i \leq p.$
Suppose $\phi^{2}_{s,t}>\sum_{a=1}^{s-1}(d_{a}-d_{s}).$
Referring to \eqref{eq2.1} and \eqref{eq2.2}, for $1 \leq i \leq s-1$
\begin{equation}
r_{i} \leq \frac{1}{x_{i}} \left( (d_{i}-t+1)d'_{t}+\sum_{h=1}^{t-1}d'_{h} \right)
+ \frac{1}{x_{i}}\sum_{k=1}^{s-1}(x_{k}-1)d_{i} = \phi^{2}_{s,t}
\label{eq3.1}
\end{equation}
and for $s \leq i \leq p$
\begin{eqnarray}
r_{i} &\leq&  (d_{i}-t+1)d'_{t}+\sum_{h=1}^{t-1}d'_{h}
+ \sum_{k=1}^{s-1}(x_{k}-1)d_{i}
\label{eq3.2.1}   \\
&\leq& (d_{s}-t+1)d'_{t}+\sum_{h=1}^{t-1}d'_{h}
+ \sum_{k=1}^{s-1}(x_{k}-1)d_{s} = \phi^{2}_{s,t},
\label{eq3.2.2}
\end{eqnarray}
where the inequalities are from Lemma~\ref{rem3}(ii)(iv) and the nonincreasing of degree sequence,
and the equalities are from Lemma~\ref{lem_x}(i).
Thus, $r_{i} \leq \phi^{2}_{s,t}$ for $1 \leq i \leq p.$
Similarly, referring to \eqref{eq2.3}, \eqref{eq2.4.3},
Lemma~\ref{rem3}(iii)(v), the nonincreasing of degree sequence, and Lemma~\ref{lem_x}(ii) we have
$r'_{j} \leq \phi^{2}_{s,t}$ for $1 \leq j \leq q.$
Hence $\rho(A^2) = \rho(C) \leq \phi_{s, t}^{2}$ by Lemma~\ref{lem4}.
\medskip

To verify the second part of Theorem~\ref{thm3}, assume that $G$ is connected.
We prove the sufficient conditions of $\rho(G) = \phi_{s, t}.$
If $s'=0$ or $t'=0$ then $G$ is biregular.
By Lemma~\ref{lem_d1d1} and Lemma~\ref{rem2}(i)(ii), $\rho(G)=\sqrt{d_{1}d'_{1}}=\phi_{s,t}$.
Suppose $s'=0$ and $t' \geq 1.$ Then $d_{1}=d_{p}$ and
$p=d_{1}'=d'_{t'} \geq d'_{t'+1}=d'_{q}.$
We take the equatable quotient matrix $E$ of $A$ with respect to the partition
$\{\{1,\ldots,p\},\{p+1,\ldots,p+t'\},\{p+t'+1,\ldots,p+q\}\}.$ Hence
$$E = \left(
        \begin{array}{ccc}
          0 & t' & d_{s}-t' \\
          p & 0 & 0 \\
          d'_{t} & 0 & 0 \\
        \end{array}
      \right).$$
The eigenvalues of $E$ are $0$ and
$\pm \sqrt{d_{s}d'_{t}+(p-d'_{t})(t'-1)} = \pm \phi_{s,t}.$
By Lemma~\ref{lemeqpa}, $\phi_{s,t}$ is also an eigenvalue of $A.$
Since $\rho(G) \leq \phi_{s,t}$ has been shown in the first part, we have $\rho(G)=\phi_{s,t}.$
Similarly for the case $s' \geq 1$ and $t'=0.$
Suppose $s' \geq 1$ and $t' \geq 1.$
Then $q=d_{1}=d_{s'} \geq d_{s'+1}=d_{p}$ and $p=d'_{1}=d'_{t'} \geq d'_{t'+1}=d'_{q}.$
We take the equatable quotient matrix $F$ of $A$ with respect to the partition
$\{\{1,\ldots,s'\},\{s'+1,\ldots,p\},\{p+1,\ldots,p+t'\},\{p+t'+1,\ldots,p+q\}\}.$ Hence
$$F = \left(
        \begin{array}{cccc}
          0 & 0 & t' & q-t' \\
          0 & 0 & t' & d_{s}-t' \\
          s' & p-s' & 0 & 0 \\
          s' & d'_{t}-s' & 0 & 0 \\
        \end{array}
      \right).$$
Then the eigenvalues of $F$ are
$$\pm \sqrt{\frac{X_{s,t}\pm\sqrt{X_{s,t}^{2}-4Y_{s,t}}}{2}}.$$
We see $\phi_{s,t}$ is an eigenvalue of $F,$ and
by Lemma~\ref{lemeqpa} $\phi_{s,t}$ is also an eigenvalue of $A.$
Hence $\rho(G)=\phi_{s,t}.$
Here we complete the proof of the sufficient conditions of $\phi_{s,t} = \rho(G).$
\medskip

To prove the necessary conditions of $\rho(G)=\phi_{s,t},$ suppose $\rho(G)=\phi_{s,t}.$
Then by Lemma~\ref{lem4} $r_{i}=r'_{j}=\phi^{2}_{s,t}$
for $1 \leq i \leq p$ and $1 \leq j \leq q.$
Let $s'<s$ and $t'<t$ be the smallest nonnegative integers
such that $d_{s'+1}=d_{s}$ and $d'_{t'+1}=d_{t},$ respectively.
We prove either $d_{1}=d_{p}$ or $q=d_{1}=d_{s'}>d_{s'+1}=d_{p}$ in the following.
The connectedness of $G$ implies $d_{s}d'_{t}>0$ so that
\begin{equation}
\phi^{2}_{s,t} > \max \left( \sum_{i=1}^{s-1}(d_{i}-d_{s}), \sum_{j=1}^{t-1}(d'_{j}-d'_{t}) \right)
\nonumber
\end{equation}
by Lemma~\ref{rem2}(iii).
Hence the equalities in \eqref{eq3.1} to \eqref{eq3.2.2} all hold.
The choose of $s'$ and the equalities in \eqref{eq3.2.2} imply that $d_{s'+1}=d_{s}=d_{p}.$
If $s'=0$ then $d_{1}=d_{p}.$ Suppose $s' \geq 1.$
For $1 \leq i \leq s',$
since $d_{i}>d_{s}$ and $\phi^{2}_{s,t} > \sum_{a=1}^{s-1}(d_{a}-d_{s})$, we have $x_{i} > 1$
by \eqref{eq2.5}.
The equalities in \eqref{eq3.1} imply $n_{ik}=d_{i}$ and then $G(u_{k}) \supseteq G(u_{i})$
by Lemma~\ref{rem3}(ii) for $1 \leq k \leq s'$ and $1 \leq i \leq s-1.$
Similarly the equalities in \eqref{eq3.2.1} imply $G(u_{k}) \supseteq G(u_{i})$
for $1 \leq k \leq s'$ and $s \leq i \leq p$ by Lemma~\ref{rem3}(ii).
That is,
\begin{equation}
G(u_{1})=G(u_{2})=\cdots=G(u_{s'}) \supseteq G(u_{i}) ~~~\text{for}~~~ s'+1 \leq i \leq p.
\nonumber
\end{equation}
Due to the connectedness of $G,$ $d_{1}=d_{s'}=q.$ The result follows.
Similarly, either $d'_{1}=d'_{q}$ or $p=d'_{1}=d'_{t'}>d'_{t'+1}=d'_{q}.$
Clearly that the graphs with those degree sequences are
$K_{s',t'}+H$ for
some biregular graph $H$ of bipartition orders $p-s'$ and $q-t'$ respectively.
Here we complete the proof for the necessary conditions of $\phi_{s,t} = \rho(G),$
and also for the the Theorem~\ref{thm3}.
\end{proof}
\medskip

\begin{rem}
Other previous results shown by  the style of the above proof
can be found in \cite{sw:04, lw:13, cls:13, hw:13}. Similar earlier results
are referred to  \cite{ra:85, bs:86, r:87, h:98, hsf:01}.
\end{rem}

\section{A few special cases of Theorem~\ref{thm3}}   \label{sapp}

In this section we study some special cases of $\phi_{s, t}$ in Theorem~\ref{thm3}.
We follow the notations in Theorem 3.3.
As $\phi_{1,1} = \sqrt{d_{1}d'_{1}}$ in Lemma~\ref{rem2}(i),
Theorem~\ref{thm3} provides another proof of
$\rho(G) \leq \sqrt{d_{1}d'_{1}}$ in Lemma~\ref{lem_d1d1}.
Applying  Theorem~\ref{thm3} and simplifying the formula $\phi_{s, t}$ in
cases $(s, t)=(1, q)$ and $(s, t)=(p, 1),$ we have the following corollary.

\begin{cor}
\begin{enumerate}
\item[(i)] $\rho(G) \leq \phi_{1,q} = \sqrt{e-(q-d_{1})d'_{q}}.$
\item[(ii)] $\rho(G) \leq \phi_{p,1} = \sqrt{e-(p-d'_{1})d_{p}}.$
\end{enumerate} \qed
\end{cor}

\medskip

We can quickly observe that
\begin{equation}
X_{p,q}=d_{p}d'_{q}+(e-pd_{p})+(e-qd'_{q})=2e-(pd_{p}+qd'_{q}-d_{p}d'_{q})
\label{eqXpqab}
\end{equation}
and
\begin{equation}
Y_{p,q}=(e-pd_{p})(e-qd'_{q}).
\label{eqYpqab}
\end{equation}
Hence we have the following corollary.
\begin{cor}     \label{corphipq}
$$
\rho(G) \leq  \sqrt{\frac{2e-(pd_{p}+qd'_{q}-d_{p}d'_{q})+\sqrt{(pd_{p}+qd'_{q}-d_{p}d'_{q})^{2}-4d_{p}d'_{q}(pq-e)}}{2}}.
$$ \qed
\end{cor}

By adding an isolated vertex if necessary, we might assume $d_p=0$
and find $\phi_{p, q}=\sqrt{e}$ from Corollary~\ref{corphipq}.
Hence Theorem~\ref{thm3} provides another proof of
$\rho(G) \leq \sqrt{e}$ in Lemma~\ref{lem_sqrte}.

\section{Proof of Conjecture~\ref{conj2}}   \label{s_conj}

When $e, p, q$ are fixed, the formula \begin{equation}\label{phipq}
\phi_{p, q}(d_p, d'_q) =\sqrt{\frac{2e-(pd_{p}+qd'_{q}-d_{p}d'_{q})+\sqrt{(pd_{p}+qd'_{q}-d_{p}d'_{q})^{2}-4d_{p}d'_{q}(pq-e)}}{2}}
\end{equation}
obtained in Corollary~\ref{corphipq} is a $2$-variable function.
The following lemma
will provide shape of the function $\phi_{p, q}(d_p, d'_q).$
\medskip

\begin{lem}     \label{lem_partial}
If $1 \leq d'_{q} \leq p-1$ and $qd'_q\leq e$ 
then
\begin{equation}
\frac{\partial \phi_{p,q}(d_{p},d'_{q})}{\partial d_{p}} < 0.
\nonumber
\end{equation}
\end{lem}
\begin{proof} Referring to  (\ref{phipq}),
it suffices to show that
\begin{align}\label{partial}
 &\frac{2e-(pd_p+qd'_q-d_pd'_q)+ \sqrt{(pd_p+qd'_q-d_pd'_q)^2-4d_pd'_q(pq-e)}}{\partial d_{p}} \nonumber \\
=&-p+d'_{q}+\frac{(pd_{p}+qd'_{q}-d_{p}d'_{q})(p-d'_{q})-2d'_{q}(pq-e)}
{\sqrt{(pd_{p}+qd'_{q}-d_{p}d'_{q})^2-4d_{p}d'_{q}(pq-e)}}
\end{align}
is negative.
If $qd'_{q}=e$ then (\ref{partial}) has negative value $2(d'_{q}-p).$
Indeed if the numerator of the fraction in  (\ref{partial}) is not positive
then (\ref{partial}) has negative value.
Thus assume  that it is positive and  $qd'_{q} < e.$
From simple computation to have the fact that
\begin{align*}
&~~\left( (pd_{p}+qd'_{q}-d_{p}d'_{q})-2d'_{q}\cdot\frac{pq-e}{p-d'_{q}} \right)^{2}
- \left( (pd_{p}+qd'_{q}-d_{p}d'_{q})^2-4d_{p}d'_{q}(pq-e) \right)
  \\
=&~~ \frac{4d'^{2}_{q}(pq-e)}{(p-d'_{q})^{2}} \cdot (qd'_{q}-e) < 0,
\end{align*}
we find that the fraction in (\ref{partial}) is strictly less than $p-d'_q$, so
the value in (\ref{partial}) is negative.
\end{proof}

\medskip

\begin{rem}   \label{rem_almostKpq}
From Example~\ref{exam_almostKpq}, if $p\leq q$ then the graphs $K_{p,q}^{[e]}=K_{p-1, q-pq+e}+N_{1, pq-e}$ and $K_{p,q}^{\{e\}}=K_{p-pq+e, q-1}+N_{pq-e, 1}$ satisfy the equalities in Theorem~\ref{thm3}. Hence $\rho(K_{p,q}^{[e]})=\phi_{p,q}(q-pq+e, p-1)$ and $\rho(K_{p,q}^{\{e\}}) =\phi_{p,q}(q-1, p-pq+e),$
the latter expanded as
\begin{equation}   \label{pqe}
\rho(K_{p,q}^{\{e\}})=\sqrt{\frac{e+\sqrt{e^2-4(q-1)(p-pq+e)(pq-e)}}{2}}
\end{equation}
by (\ref{phipq}).
\end{rem}

\medskip

\begin{lem}     \label{lem_conj2}
Suppose $0 < pq-e < \min(p,q)$, $1\leq d_p\leq q-1$, $1\leq d'_q\leq p-1$ and
\begin{equation}
d_{p}+d'_{q}=e-(p-1)(q-1).
\label{deg_relation}
\end{equation}
Then
\begin{equation}
\phi_{p,q}(d_{p},d'_{q}) \leq \rho(K_{p,q}^{\{e\}}).
\nonumber
\end{equation}
\end{lem}

\begin{proof} From symmetry, we can assume $p\leq q.$
Referring to (\ref{phipq})  and (\ref{pqe}), we only need to show that
\begin{align}
&e-(pd_{p}+qd'_{q}-d_{p}d'_{q})+ \sqrt{(pd_{p}+qd'_{q}-d_{p}d'_{q})^{2}-4d_pd'_q(pq-e)} \label{1st}\\ \leq &
\sqrt{e^2-4(q-1)(p-pq+e)(pq-e)}. \label{last}
\end{align}

From \eqref{deg_relation}, we have
\begin{equation}
e - (pd_{p}+qd'_{q}-d_{p}d'_{q}) = (p-d'_{q}-1)(q-d_{p}-1) \geq 0
\label{ineq_1}
\end{equation}
and
\begin{align}
d_{p}d'_{q}= &\frac{(d_p+d'_q)^2-(d_p-d'_q)^2}{4}\nonumber \\
       \geq  &\frac{(e-(p-1)(q-1))^2-(q-1-(e-(p-1)(q-1)-(q-1)))^2}{4}\nonumber \\
        =    &(q-1)(p-pq+e).
\label{ineq_2}
\end{align}
Hence the equation (\ref{1st}) is at most
\begin{equation}\label{2nd}
e-(pd_{p}+qd'_{q}-d_{p}d'_{q})+\sqrt{(pd_{p}+qd'_{q}-d_{p}d'_{q})^{2}-4(q-1)(p-pq+e)(pq-e)}.
\end{equation}
By setting
$a = e-(pd_{p}+qd'_{q}-d_{p}d'_{q})$ and $b=4(q-1)(p-pq+e)(pq-e)$
and using the fact that \begin{equation}
\sqrt{e^2-b}-\sqrt{(e-a)^2-b} \geq \sqrt{e^2}-\sqrt{(e-a)^2} = a
\label{ineq_concave}
\end{equation}
from the concave property of the function $y=\sqrt{x}$,  we find the value in (\ref{2nd}) is at most
that in (\ref{last})
and the result follows.
\end{proof}

\medskip

\noindent{\bf The proof of Conjecture~\ref{conj2}}
\begin{proof}
By Theorem~\ref{thm3}, $\rho(G) \leq \phi_{p,q}(d_{p},d'_{q}).$
Note that the assumption $0<pq-e<\min(p,q)$ implies $1 \leq d_{p} \leq q-1$ and $1 \leq d'_{q} \leq p-1.$
Let $e_{p}=e-(p-1)(q-1)-d'_{q}.$
Clearly that $1 \leq e_{p} \leq d_{p}.$
By  Lemma~\ref{lem_partial}, $\phi_{p,q}(d_{p},d'_{q}) \leq \phi_{p,q}(e_{p},d'_{q}).$
With $e_p$ in the role of $d_p$ in Lemma~\ref{lem_conj2},
we have
$\phi_{p,q}(e_{p},d'_{q})\leq \rho(K_{p, q}^{\{e\}}).$ This completes the proof.
\end{proof}

\section{Concluding Remark}\label{concluding}

We give a series of sharp upper bounds for the spectral radius of a bipartite graph in Theorem~\ref{thm3}.
 One of these upper bounds can be presented only by five variables: the number $e$ of edges,   bipartition orders $p$ and $q$, and the minimal degrees $d_p$ and $d'_q$ in the corresponding partite sets as shown in Corollary~\ref{corphipq}.
We apply this bound when three variables $e, p, q$ are fixed  to prove Conjecture~\ref{conj2}, a refinement of Conjecture~\ref{conj1} in the assumption that $pq-e < \min(p, q).$
To conclude the paper we propose the following general refinement of Conjecture~\ref{conj1}.

\begin{conj}   \label{conj_concluding}
Let $G \in \mathcal{K}(p,q,e).$ Then
\begin{equation}
\rho(G) \leq \rho(K_{s, t}^{\{e\}})
\nonumber
\end{equation}
for some positive integers $s \leq p$ and $t \leq q$
such that $0 \leq st-e \leq \min(s,t).$
\end{conj}

\medskip

We believe that the function $\phi_{p,q}(d_{p},d'_{q})$ in \eqref{phipq} will still play an important role in solving Conjecture~\ref{conj_concluding}.
Two of the key points might be to investigate the shape of the $4$-variable function $\phi_{p,q}(d_{p},d'_{q})$
with variables $p,q,d_{p},d'_{q}$, and to check that for which sequence $s$, $t$, $d_s$, $d'_t$
such that
$s \leq p$ and $t \leq q$ and $0 \leq st-e \leq \min(s,t),$
 there exists a
bipartite graph $H$ with $e$ edges whose spectral radius satisfying $\rho(H)=\phi_{s,t}(d_{s},d'_{t})$,
where  $s,t$ are the bipartition orders of $H$ and $d_s$ and $d'_t$ are corresponding minimum degrees.

\section*{Acknowledgments}
This research is supported by the National Science Council of Taiwan R.O.C.
under the project NSC 102-2115-M-009-005-MY3.

\end{document}